\documentclass[12pt]{amsart}
\usepackage{amsthm,amssymb,amsrefs,a4wide,cancel}
\newtheorem{theorem}{Theorem}

\newtheorem{corollary}[theorem]{Corollary}
\newtheorem{remark}[theorem]{Remark}
\newtheorem{proposition}[theorem]{Proposition}

\newcommand{\ov}[1]{\overline{#1}}

\title{A note on semicentral Idempotents}
\author{Christian Lomp}
\address{Department of Mathematics of the Faculty of Science and  Center of Mathematics, University of Porto, Rua Campo Alegre, 687, 4169-007 Porto, Portugal}
\email{clomp@fc.up.pt}
\author{Jerzy Matczuk}
\address{Institute of Mathematics, Warsaw University
Banacha 2, 02-097 Warsaw, Poland }
\email{jmatczuk@mimuw.edu.pl}
\thanks{This work was done during a visit of the first author to the University of Warsaw. He would like to thank the second author and all members of the Mathematics Department for their warm hospitality.  The first author would like to thank WCMCS for financial support. Moreover the first author was partially supported by CMUP (UID/MAT/00144/2013), which is funded by FCT (Portugal) with national (MEC) and European structural funds (FEDER), under the partnership agreement PT2020.}

\begin{document}
\maketitle

\begin{abstract}
In this note we answer the question raised by Han et al. in \cite{HanLeePark} whether an idempotent isomorphic to a semicentral idempotent is itself semicentral. We show that rings with this property are precisely the Dedekind-finite rings. An application to module theory is given.
\end{abstract}

\section{Introduction}
Let $R$ be a unital associative ring with unit.  Any idempotent $e\in R$ leads to a representation of $R$ as a generalized matrix ring. An idempotent $e$ is called left (respectively right) {\bf semicentral} if $(1-e)Re=0$
(respectively $eR(1-e)=0$). Hence right (respectively left) semicentral idempotents lead to a representation of $R$ as an upper (lower) generalized triangular matrix ring (see also \cite{BirkenmeierParkRizvi}). Clearly an idempotent is central if and only if it is left and right semicentral.

The purpose of this note is to answer \cite{HanLeePark}*{Question 1} which asks whether any idempotent isomorphic to a semicentral idempotent is itself semicentral. We will show that this property characterises Dedekind-finite rings.

\section{Isomorphic semicentral idempotents}
Two idempotents $e$ and $f$ are said to be {\bf isomorphic} if $eR\simeq fR$ as right $R$-modules.
It is not difficult to see (but the reader may also consult \cite{Lam}*{21.20}) that two idempotents $e$ and $f$ are isomorphic if and only if  there are elements $a,b \in R$ such that $e=ab$ and $f=ba$.

\begin{remark}\label{Jurek}
 Suppose that any idempotent
 that is isomorphic to a left semicentral idempotent of a ring $R$ is itself left semicentral. Then $R$ is Dedekind-finite.
\end{remark}
\begin{proof}
If $ab=1$ for some $a,b \in R$, then $e=ba$ is an idempotent that is isomorphic to the central idempotent $1$. By hypothesis $e$ is left semicentral. Thus $1=abab=(ae)b=(eae)b=e(abab)=e$.
\end{proof}
Note that a similar proof for Remark \ref{Jurek} shows that $R$ is Dedekind-finite, provided any idempotent isomorphic to a right semicentral idempotent is itself right semicentral.

The following remark is well-known and can be found for example in \cite{Lam}*{4.8} and \cite{BreazCalugareanuSchultz}*{2.4}, but we include a short proof for sake of completeness:
\begin{remark}\label{Jacobson-Dedekind}\label{Corner-Dedekind}
Let $R$ be a Dedekind-finite ring with Jacobson radical $J$, then $R/J$ is Dedekind-finite and
$eRe$ is Dedekind-finite, for any non-zero idempotent $e\in R$.
\end{remark}

\begin{proof}
Suppose $\ov{a}\ov{b}=\ov{1}$ in $\ov{R}=R/J$ and $a,b\in R$. Then $1-ab\in J$ implies that $ab$ is invertible in $R$. Thus there exists $u\in U(R)$ such that $abu=1$. As $R$ is Dedekind-finite, also $1=bua$ holds.
Hence $\ov{1}=\ov{bu}\ov{a}$ implies $\ov{b}=\ov{bu}\ov{a}\ov{b}=\ov{bu}$, i.e. $\ov{b}\ov{a}=1$.

Now suppose $S=eRe$ and $f=1-e$ for some idempotent $e\in R$. If $a,b \in S$ with $ab=1_S = e$, then $(a+f)(b+f)=ab+f=e+f=1$. Since $R$ is Dedekind-finite, $1=(b+f)(a+f)=ba+f$ and $1_S=e=ba$ follows.
\end{proof}

\begin{remark}\label{isomorphicIdempotents}
Suppose that $e$ and $f$ are isomorphic idempotents such that $e=ab$ and $f=ba$, for $a,b\in R$. Since $f=fba$, we also have $(fb)(ea)=fbaba=f^3=f$. Similarly $e=eab$ shows $(ea)(fb)=eabab=e^3=e$. Therefore we may replace $a$ by $ea$ and $b$ by $fb$ and assume that $a\in eR$ and $b\in fR$.
\end{remark}

\begin{remark}\label{central}
Any left or right semicentral idempotent in a semiprime ring is central.
\end{remark}

\begin{proof}
If $e$ is a left semicentral idempotent in a semiprime ring $R$, then $(1-e)Re=0$ and $ReR(1-e)R$ is a nilpotent ideal.
Hence $eR(1-e)=0$ and $e$ is central.
\end{proof}

\begin{proposition}\label{conjugated}
Any idempotent isomorphic to a left or right semicentral idempotent in a Dedekind-finite ring is already conjugated to it.
\end{proposition}

\begin{proof}
Let $e$ be a left semicentral idempotent in a Dedekind-finite ring $R$ and let $f$ be an idempotent  that is isomorphic to $e$. By Remark \ref{Jacobson-Dedekind}, $\ov{R}=R/J$ is Dedekind-finite and by Remark \ref{central} $\ov{e}$ is central. Set $S:=\ov{e}\ov{R}$. Since $e$ and $f$ are isomorphic in $R$, also $\ov{e}$ and  $\ov{f}$ are isomorphic in $\ov{R}$.  By Remark \ref{isomorphicIdempotents}, there exist $\ov{a}\in S$ and $\ov{b}\in \ov{f}\ov{R}$ such that $ \ov{e}=\ov{a}\ov{b}$ and $\ov{f} =\ov{b}\ov{a}$.
Then using that $\ov{e}$ is central we have $\ov{f}=\ov{f}^2 = \ov{b}\ov{e}\ov{a}=\ov{e}\ov{f}\in S$ and therefore also $\ov{b}\in S$.
As $1_S=\ov{e}=\ov{a}\ov{b}$ and $S=\ov{e}\ov{R}$ is Dedekind-finite by Remark \ref{Corner-Dedekind}, we have $\ov{f}=\ov{b}\ov{a}=1_S=\ov{e}$. Thus there exist $x\in J$ such that $f=e+x$. By \cite{KanwarLeroyMatczuk}*{Proposition 9}, $ue=fu$ for $u=1+x(2e-1)$ holds. Since $x\in J$, $u$ is invertible and hence $f=ueu^{-1}$.
\end{proof}

\begin{theorem}\label{Dedekind-Finite-Theorem}
The following statements are equivalent for a ring $R$:
\begin{enumerate}
\item[(a)] Any idempotent that is isomorphic to a left semicentral idempotent is itself left semicentral.
\item[(a')] Any idempotent that is isomorphic to a right semicentral idempotent is itself right semicentral.
\item[(b)] Any idempotent that is isomorphic to a left or right semicentral idempotent is conjugated to it.
\item[(c)] $R$ is Dedekind-finite.
\end{enumerate}
\end{theorem}

\begin{proof}
Remark \ref{Jurek} proves $(a) \Rightarrow (c)$ and Proposition \ref{conjugated} shows $(c) \Rightarrow (b)$. The implication $(b) \Rightarrow (a)$ is clear, because if $\alpha:R\rightarrow R$ is  any ring isomorphism,
then $(1-e)Re=0$ implies $(1-\alpha(e))R\alpha(e)=0$, for any $e\in R$.

\end{proof}

%\section{Application to module theory}

Let $M$ be a non-zero right $R$-module with endomorphism ring $S$. It is well-known, that   any idempotent $e$ of $S$ corresponds to a direct summand $\mathrm{Im}(e)$ of $M$ and that this correspondence is bijective. In particular, isomorphic direct summands of $M$ correspond to isomorphic idempotents in $S$.  From  \cite{LamEx}*{21.16} it is known, that two idempotents $e$ and $f$ are conjugate in $S$ if and only if $e$ and $f$ as well as $1-e$ and $1-f$ are isomorphic. If $e$ is an idempotent corresponding to a direct summand $D$ of $M=D\oplus D'$, then $1-e$ corresponds to $D'\simeq M/D$. Hence two idempotents $e$ and $f$ in $S$  are conjugate in $S$ if and only if for their corresponding direct summands $D_1$ and $D_2$, $D_1\simeq D_2$ and $M/D_1\simeq M/D_2$ holds. Moreover, it is easy to see that an idempotent  $e\in S$, with corresponding direct summand $D=\mathrm{Im}(e)$, is left semicentral if and only if $\mathrm{Hom}(D,M/D)=0$.

 Using the above we obtain the following module theoretic version of Theorem \ref{Dedekind-Finite-Theorem}:

 \begin{corollary}
The following statements are equivalent for a right $R$-module $M$.
\begin{enumerate}
\item[(a)] For any isomorphic direct summands $D_1$ and $D_2$ of $M$, if $\mathrm{Hom}(D_1,M/D_1)=0$, then $\mathrm{Hom}(D_2,M/D_2)=0$
\item[(a')] For any isomorphic direct summands $D_1$ and $D_2$ of $M$, if $\mathrm{Hom}(M/D_1,D_1)=0$, then $\mathrm{Hom}(M/D_2,D_2)=0$
\item[(b)] For any isomorphic direct summands $D_1$ and $D_2$ of $M$, if $\mathrm{Hom}(D_1,M/D_1)=0$ or $\mathrm{Hom}(M/D_1,D_1)=0$, then $M/D_1\simeq M/D_2$.
\item[(c)] $\mathrm{End}(M)$ is Dedekind-finite.
\end{enumerate}
\end{corollary}

\end{document}